\documentclass[reqno,11pt]{amsart}

\usepackage{pinlabel}
\input{epsf.tex}
\usepackage{latexsym,amsfonts,amssymb,verbatim,mathrsfs,amsthm}
\usepackage{amsmath,amsthm,amssymb,latexsym,graphics,textcomp}
\usepackage{eucal,eufrak}
\usepackage{colonequals}
\usepackage{graphicx}
\usepackage{url}
\usepackage{xcolor, wasysym}
\usepackage{import}
\usepackage{transparent}
\input{xy}
\xyoption{all}

\usepackage[colorlinks,citecolor=blue!70!black,linkcolor=red!60!black]{hyperref}
\usepackage[nameinlink]{cleveref}

\theoremstyle{plain}
\newtheorem{theorem}{Theorem}[section]
\newtheorem{lemma}[theorem]{Lemma}
\newtheorem{corollary}[theorem]{Corollary}

\newtheorem{proposition}[theorem]{Proposition}

\theoremstyle{definition}
\newtheorem{remark}{Remark}

\newtheorem{question}{Question}

\def\1{\mathbf 1}

\def\Sym{\mathrm{Sym}}

\def\ZZ{\mathbb Z}

\def\cR{\mathcal{R}}

\def\id{\mathrm{id}}

\DeclareMathOperator{\Map}{Map}
\DeclareMathOperator{\PMap}{PMap}

\newcounter{jcomments}

\newcounter{gcomments}

\newcounter{ccomments}

\title{Isomorphisms and commensurability of surface Houghton groups}
\author{Javier Aramayona}
\author{George Domat}
\author{Christopher J. Leininger}
\thanks{J. A.~was supported by grant PID2021-126254NB-I00 and by the Severo Ochoa
award CEX2019-000904-S, funded by MCIN/AEI/10.13039/501100011033. G. D.~was supported by NSF DMS-2303262. C. J. L.~was supported by NSF DMS-2305286.}
\begin{document}

\begin{abstract}
We classify surface Houghton groups, as well as their pure subgroups, up to isomorphism, commensurability, and quasi-isometry. 
\end{abstract}

\maketitle

\section{Introduction}

Surface Houghton groups, introduced in \cite{ABKL2023}, are certain countable subgroups of mapping class groups of infinite-type surfaces with finitely many ends, all of them non-planar. 

A first piece of motivation for studying surface Houghton groups comes from the fact that they contain, up to conjugation, all {\em end periodic homeomorphisms} \cite[Section 2.3]{FKLL2023}.  Such homeomorphisms admit a Nielsen-Thurston type classification \cite{CC-book} and their properties are tightly connected to the theory of foliations, pseudo-Anosov flows, and hyperbolic geometry of 3-manifolds; see e.g.~\cite{FenleyCT,FKLL2023,FKLL2023b,LMT2023}. In addition, surface Houghton groups are similar in spirit to the classical Houghton groups \cite{Houghton} and their {\em braided} counterparts \cite{Funar}; in fact, in \cite{ABKL2023} it was proved that surface Houghton groups enjoy the same finiteness properties as (braided) Houghton groups \cite{GLU2022}. Finally, surface Houghton groups are strongly related to the asymptotically rigid mapping class groups of Cantor manifolds \cite{ABFPW2021,FKS2012,GLU2022}, and associated Higman-Thompson groups.

The  definition of surface Houghton groups is somewhat involved, so we postpone details to \Cref{sec:prelim} and give an abridged overview here. Let $\Sigma_r$ be the connected, orientable surface with empty boundary and exactly $r$ ends, all of which are non-planar. We view $\Sigma_r$ as constructed from a compact surface of genus $g$ with $r$ boundary components by inductively gluing copies of a surface of genus $h$ with two boundary components, and then taking the union of the surfaces obtained at each step. 
The {\em surface Houghton group} $B(g,h,r)$ defined by the above data is the subgroup of the mapping class group $\Map(\Sigma_r)$ whose elements are {\em eventually rigid}, in the sense of \Cref{sec:prelim}. We remark that the family of surface Houghton groups constructed here is more general than that of \cite{ABKL2023}, which corresponds to the case $g=0$ and $h=1$.

In this note, we describe precisely the commensurability and isomorphism classification of surface Houghton groups and their pure subgroups.
More concretely, we will show the following.

\begin{theorem}\label{thm:main} For $g,g' \geq 0$, $h,h'\geq 1$, and $r,r' \geq 2$, the groups $B(g,h,r)$ and $B(g',h',r')$:
\begin{enumerate}
    \item are commensurable if and only if $r= r'$.
    \item have isomorphic pure subgroups if and only if $r=r'$ and $h=h'$.
    \item are isomorphic if and only if $r = r', h=h'$, and there exists $n \in \mathbb Z$ such that $g' \equiv g + nh$ mod $r$.
\end{enumerate}
\end{theorem}

We note that, for all $g$ and $h$, the group $B(g,h,1)$ is isomorphic to the group of compactly supported mapping classes of $\Sigma_{1}$, so \Cref{thm:main} is only interesting for $r,r' \geq 2$. 

\begin{remark}
    Point (3) of the above theorem may seem slightly unintuitive as we are only changing a ``compact" part of the surface. We offer the following example to keep in mind in order to elucidate the proofs below. See \Cref{fig:example} for an accompanying picture. Consider the groups $B(0,2,2)$ and $B(1,2,2)$. In $B(0,2,2)$ there is an involution, $\rho$, that swaps the two ends of the surface and has exactly two fixed points. This map is ``eventually rigid" exactly outside of an annulus denoted $W$, which is also exactly the core, $C$, for $B(0,2,2)$. We say that $W$ is a \emph{suited surface} for $\rho$. Note that $W$ has even (zero) genus. Now, suppose to the contrary, that we had an isomorphism $B(0,2,2) \cong B(1,2,2)$. The first step in our proof is to modify a proof from \cite{BDR2020} in order to see that this isomorphism must be realized by conjugation in $\Map(\Sigma_{2})$. Then the image of $\rho$ under this isomorphism, $\rho' \in B(1,2,2)$, must also have exactly two fixed points and be ``eventually rigid" for the structure determined by the triple $(1,2,2)$. However, note that this implies that $\rho'$ has a suited surface, $W'$, that must have odd genus. Indeed, $W'$ is made up of a core $C'$ of genus one together with a finite number of pieces, each of which have genus two. However, then $\rho'$ must have at least four fixed points, a contradiction. 

     \begin{figure}[ht!]
 \centering
	    \def\svgwidth{\textwidth}
		    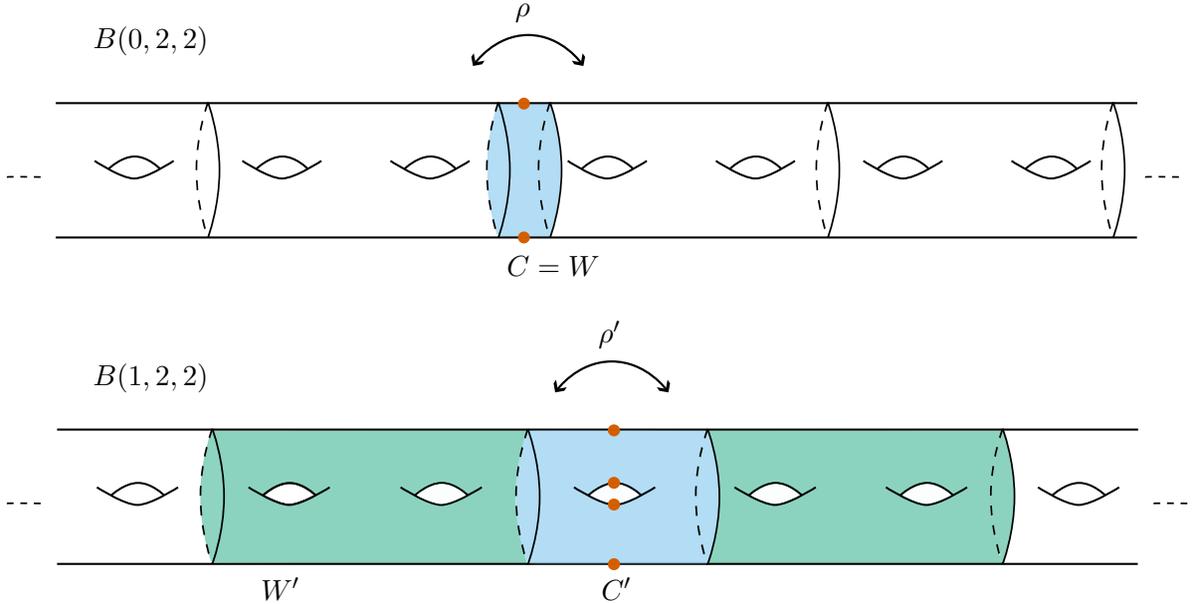
		    \caption{An example demonstrating that $B(0,2,2)$ cannot be isomorphic to $B(1,2,2)$. On the top surface, $C$ is the core (denoted in blue) for $B(0,2,2)$. This coincides with a suited subsurface, $W$, for the involution $\rho \in B(0,2,2)$. On the bottom surface, $\rho'$ is a potential image of $\rho$ under an alleged isomorphism, $C'$ (in blue) is the core, and $W'$ (in green) is a suited subsurface for $\rho'$. Fixed points of the two maps are denoted in orange. Note that $\rho$ has exactly two fixed points and $\rho'$ must have at least four.}
	    \label{fig:example}
\end{figure}

\end{remark}



A consequence of part (1) of \Cref{thm:main} and the main result of \cite{ABKL2023} is the following: 

\begin{corollary}\label{cor:finprop}
    For all $g\ge 0$, $h\ge 1$, and $r\ge 2$, the group $B(g,h,r)$ is of type $F_{r-1}$ but not of type $FP_r$. 
\end{corollary}

Next, a result of Alonso \cite{Alonso} asserts that properties $F_n$ and $FP_n$ are invariant under quasi-isometries. Therefore, we deduce the following corollary. 

\begin{corollary}
    For $g,g' \geq 0$, $h,h'\geq 1$, and $r,r' \geq 2$, the groups $B(g,h,r)$ and $B(g',h',r')$
    are quasi-isometric if and only if $r= r'$.
    \label{cor:qi}
\end{corollary}


It was mentioned above that surface Houghton groups are analogs of the asymptotically rigid mapping class groups of Cantor surfaces \cite{ABFPW2021,FKS2012,GLU2022}, which can in turn be regarded as surface versions of Higman-Thompson groups \cite[Section 5]{ABFPW2021}. Higman-Thompson groups are known to satisfy an algebraic rigidity theorem akin to \Cref{thm:main} \cite{Pardo2011}; thus, in light of \Cref{thm:main} and the analogous result for Higman-Thompson groups, a natural question is as follows. 

\begin{question}
   Classify asymptotically rigid mapping class groups of Cantor surfaces up to commensurability/quasi-isometry/isomorphism. 
\end{question}

In regards to the question above, we remark that asymptotically rigid mapping class groups of infinite-genus Cantor surfaces have no finite-index subgroups; moreover, their pure subgroups always coincide with the compactly supported mapping class group \cite[Proposition 4.3]{ABFPW2021}. 

Surface Houghton groups can also be thought of as a two-dimensional analog of classical Houghton groups. The Houghton groups, $H_{r}$, can be thought of as the ``asymptotically rigid" homeomorphisms of the space consisting of $r>0$ disjoint convergent sequences of isolated points. The finiteness properties of these groups were computed in \cite{Brown1987} and one can deduce that they are never isomorphic for different values of $r$. Similar to the surface Houghton groups, one could define Houghton groups $H(g,h,r)$ for $g\geq 0$ and $h,r>0$ and ask if results similar to \Cref{thm:main} hold. The general approach taken in this note should work for these Houghton groups. The key piece that is missing and requires proof is a version of \Cref{prop:commen}. That is, one would need to know verify that an isomorphism of Houghton groups is always induced by conjugation in the larger homeomorphism group.

\medskip

\noindent{\bf Acknowledgements.} J.A. is grateful to Rice University, and particularly to C.J.L, for their hospitality.  The authors are grateful to Anthony Genevois for pointing out Corollary \ref{cor:qi} and thank the referee for several helpful comments.

\section{Preliminaries} \label{sec:prelim}

Throughout, we let $\Sigma_r$ be the connected, orientable surface with empty boundary and exactly $r$ ends, all of which are accumulated by genus. For each $h > 0$, let $Y^h$ be a surface of genus $h$ with two boundary components, denoted $\partial_+Y^h$ and $\partial_- Y^h$.  We fix, once and for all, a homeomorphism $\lambda \colon \partial_- Y^h \to \partial_+ Y^h$ which we call the {\em swapping homeomorphism}. 

A {\em $(g,h)$--rigid structure} on $\Sigma_r$ is a decomposition of $\Sigma_r$ into subsurfaces,
\[ \Sigma_r = C \cup \bigcup_{j \in J} Y_j \]
where $C$ is a connected subsurface of genus $g$ with $r$ boundary components called the {\em core}, $J$ is some countable set, and each $Y_j$ is a subsurface of genus $h$ with two boundary components, called a {\em piece}, which comes equipped with a fixed {\em marking homeomorphism} $\phi_j \colon Y_j \to Y^h$ (which we also use to label the boundary components $\partial_\pm Y_j = \phi_j^{-1}(\partial_\pm Y^h)$). We further require this data to satisfy the following conditions. 
\begin{itemize}
  \item All subsurfaces in the decomposition have disjoint interiors;
  \item For all $j \in J$, $Y_j \cap C = \emptyset$ or $Y_j \cap C = \partial_- Y_j$ is a component of $\partial C$;
  \item For all $i \neq j \in J$, $Y_i \cap Y_j = \emptyset$, or $Y_i \cap Y_j = \partial_- Y_i = \partial_+ Y_j$ and
  \[\phi_j|_{\partial_+Y_j}^{-1} \circ \lambda \circ \phi_i|_{\partial_-Y_i} = \id_{\partial_- Y_i}, \]
  after possibly swapping the roles of $i$ and $j$.
\end{itemize}
A {\em suited subsurface} $Z \subset \Sigma_r$ is a connected union of the core and a finite set of pieces.

Fixing a $(g,h)$--rigid structure $\cR_{g,h,r}$ on $\Sigma_r$, we define the {\em surface Houghton group} of $\cR_{g,h,r}$ to be the subgroup of $\Map(\Sigma_r)$ consisting of (isotopy classes of) orientation preserving homeomorphisms $f \colon \Sigma_r \to \Sigma_r$ such that for some suited subsurface $Z$, $f(Z)$ is another suited subsurface and for any piece $Y_j \subset \overline{\Sigma - Z}$ there is a piece $Y_i$ so that $f(Y_j)=Y_i$ and $f|_{Y_j} = \phi_i^{-1} \phi_j$.  Any two $(g,h)$--rigid structures on $\Sigma_r$ clearly define conjugate surface Houghton groups, and we let $B(g,h,r)$ denote any such subgroup.

The pure subgroup $\PMap(\Sigma_r)$ is the index $r!$ subgroup of $\Map(\Sigma_r)$ consisting of elements that fix each end of $\Sigma_r$.  Let $\Phi:\PMap(\Sigma_{r}) \rightarrow \ZZ^{r-1}$ be the surjective homomorphism defined in \cite{APV2020} which effectively ``counts" the shifting of genus.  We write $\Map_c(\Sigma_r)$ to denote the subgroup of $\PMap(\Sigma_r)$ consisting of compactly supported mapping classes.

We define the pure subgroup $PB(g,h,r) = \PMap(\Sigma_r) \cap B(g,h,r)$.  The restriction of $\Phi$ to $PB(g,h,r)$ maps to $(h \mathbb Z)^{r-1}$ and the kernel is precisely $\Map_c(\Sigma_r)$. Note that $\Map_c(\Sigma_{r})$ is contained in $PB(g,h,r)$ for all choices of $g$ and $h$.  The following is proved in \cite{ABKL2023} for $B(0,1,r)$, but the same argument applies.  Here, $\Sym(r)$ is the symmetric group on $r$ elements.
\begin{proposition} \label{prop:compactly supported}
For every $(g,h,r)$, there is a surjective homomorphism $\varphi \colon B(g,h,r) \to \mathbb Z^{r-1} \rtimes \Sym(r)$. Furthermore, every finite quotient factors through $\varphi$, and $\ker(\varphi) = \Map_c(\Sigma_r)$.  \qed
\end{proposition}
Composing the homomorphism $\varphi$ with the further quotient to $\Sym(r)$ describes the action of $B(g,h,r)$ on the ends, and so has kernel precisely $PB(g,h,r)$.  In addition, the restriction of $\varphi$ to $PB(g,h,r)$ maps to $\mathbb Z^{r-1}$, and this is precisely the homomorphism $\tfrac1h\Phi$.

\section{Engulfing}

Suppose we have a $(0,1)$--rigid structure $\cR_{0,1,r}$ on $\Sigma_r$.  For any $g \geq 0$ and $h \geq 1$, we can define a $(g,h)$--rigid structure from $\cR_{0,1,r}$ as follows.

First, let $C'$ be the connected union of $C$ (the core of $\cR_{0,1,r}$) together with $g$ pieces of $\cR_{0,1,r}$.  The pieces of $\cR_{g,h,r}$ are then determined by this choice of $C'$ so that they are connected unions of the pieces of $\cR_{0,1,r}$.  See \Cref{fig:engulfex}.  The homeomorphisms from a piece of $\cR(g,h,r)$ to $Y^h$ are constructed from the chosen homeomorphisms of pieces of $\cR(0,1,r)$ to $Y^1$, by viewing $Y^h$ as decomposed into the union of $h$ subsurfaces homeomorphic to $Y^1$. We note that some care must be taken in defining the preferred homeomorphism to $Y^h$ to ensure compatibility with the swapping homeomorphism, as required in the last condition of the definition of a rigid structure.  If $\cR_{g,h,r}$ is constructed in this way, we say that $\cR_{g,h,r}$ is {\em engulfed} by $\cR_{0,1,r}$. 

 \begin{figure}[ht!]
 \centering
	    \def\svgwidth{\textwidth}
\begingroup%
  \makeatletter%
  \providecommand\color[2][]{%
    \errmessage{(Inkscape) Color is used for the text in Inkscape, but the package 'color.sty' is not loaded}%
    \renewcommand\color[2][]{}%
  }%
  \providecommand\transparent[1]{%
    \errmessage{(Inkscape) Transparency is used (non-zero) for the text in Inkscape, but the package 'transparent.sty' is not loaded}%
    \renewcommand\transparent[1]{}%
  }%
  \providecommand\rotatebox[2]{#2}%
  \newcommand*\fsize{\dimexpr\f@size pt\relax}%
  \newcommand*\lineheight[1]{\fontsize{\fsize}{#1\fsize}\selectfont}%
  \ifx\svgwidth\undefined%
    \setlength{\unitlength}{455.34339862bp}%
    \ifx\svgscale\undefined%
      \relax%
    \else%
      \setlength{\unitlength}{\unitlength * \real{\svgscale}}%
    \fi%
  \else%
    \setlength{\unitlength}{\svgwidth}%
  \fi%
  \global\let\svgwidth\undefined%
  \global\let\svgscale\undefined%
  \makeatother%
  \begin{picture}(1,0.48463525)%
    \lineheight{1}%
    \setlength\tabcolsep{0pt}%
    \put(0,0){\includegraphics[width=\unitlength,page=1]{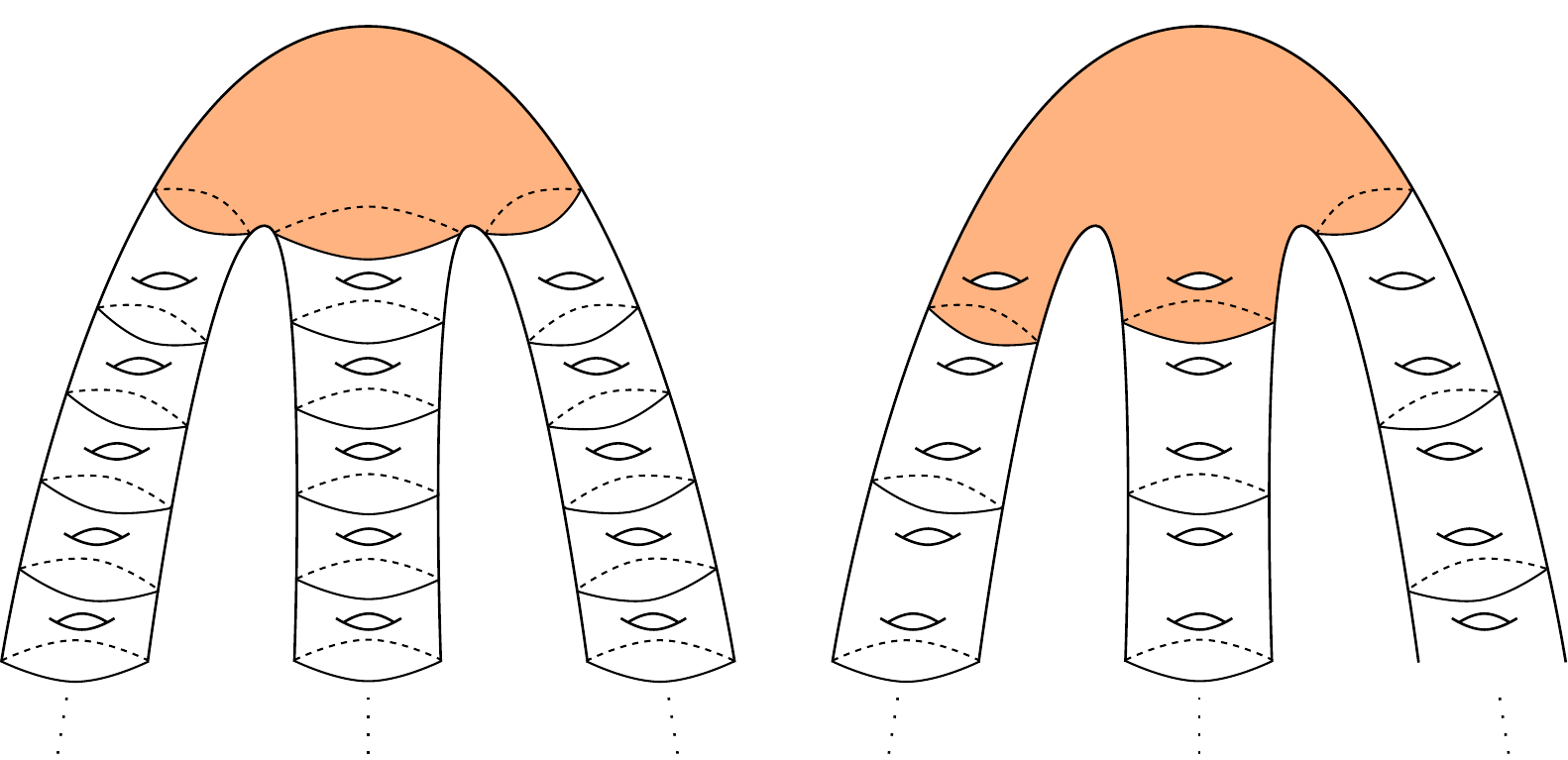}}%
    \put(0.22848835,0.38850051){\color[rgb]{0,0,0}\transparent{0.99230802}\makebox(0,0)[lt]{\lineheight{1.25}\smash{\begin{tabular}[t]{l}$C$\end{tabular}}}}%
    \put(0.76077595,0.38801086){\color[rgb]{0,0,0}\transparent{0.99230802}\makebox(0,0)[lt]{\lineheight{1.25}\smash{\begin{tabular}[t]{l}$C'$\end{tabular}}}}%
    \put(0.05285398,0.46030255){\color[rgb]{0,0,0}\transparent{0.99230802}\makebox(0,0)[lt]{\lineheight{0}\smash{\begin{tabular}[t]{l}$\mathcal{R}_{0,1,3}$\end{tabular}}}}%
    \put(0.57220434,0.46030255){\color[rgb]{0,0,0}\transparent{0.99230802}\makebox(0,0)[lt]{\lineheight{0}\smash{\begin{tabular}[t]{l}$\mathcal{R}_{2,2,3}$\end{tabular}}}}%
  \end{picture}%
\endgroup%

		    \caption{An example of a rigid structure $\cR_{2,2,3}$ (right) that is engulfed by the rigid structure $\cR_{0,1,3}$ (left). The two cores are shaded orange and the curves that are drawn are the boundaries of the pieces.}
	    \label{fig:engulfex}
\end{figure}
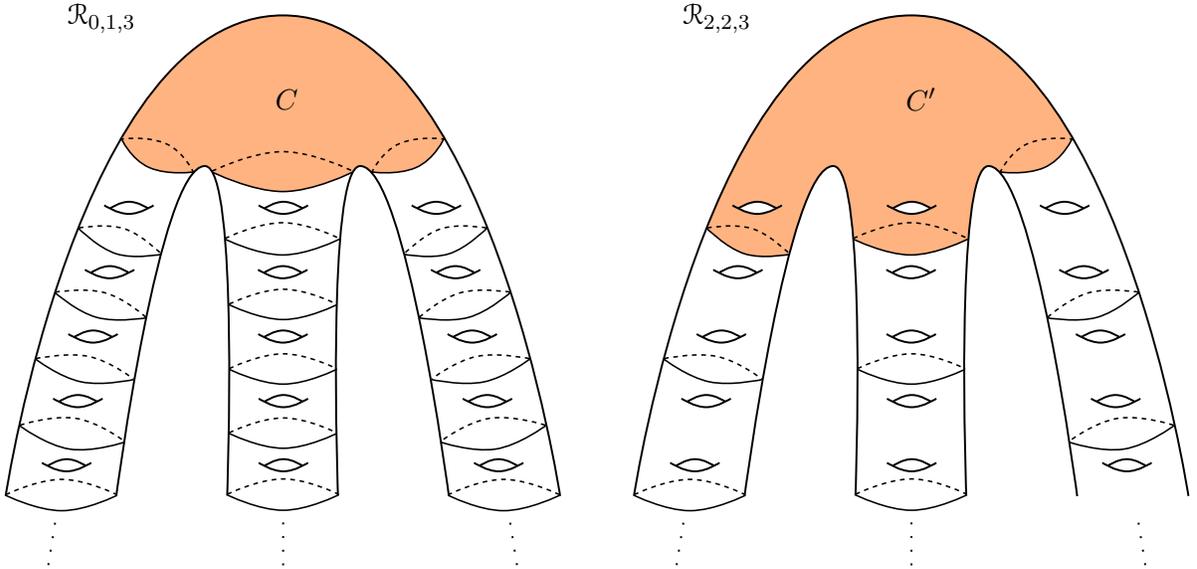

The terminology ``engulfing" is chosen to reflect the following behavior of the associated groups.
\begin{lemma} \label{lem:engulf}
    If $\cR_{g,h,r}$ is engulfed by $\cR_{0,1,r}$, then $B(g,h,r)$ is a finite index subgroup of $B(0,1,r)$.
\end{lemma}
\begin{remark}
    As a consequence of the lemma, it also follows that for arbitrary $(g,h)$--rigid and $(0,1)$-rigid structures, the associated group $B(g,h,r)$ is {\em conjugate} into the associated group $B(0,1,r)$ (in the lemma, there is no conjugation).
\end{remark}

\begin{proof} We first observe that any suited subsurface for $\cR_{g,h,r}$ is also a suited subsurface for $\cR_{0,1,r}$. Since the chosen homeomorphisms from pieces of $\cR_{g,h,r}$ to $Y^h$  are constructed from the homeomorphisms from pieces of $\cR_{0,1,r}$ to $Y^1$, it follows that $B(g,h,r) < B(0,1,r)$. To see that it has finite index, it suffices to show that $PB(g,h,r) < PB(0,1,r)$ has finite index.  For this, we can label each piece of $\cR_{0,1,r}$ inside a piece $Y_k'$ of $\cR_{g,h,r}$ with the numbers $1,\ldots,h$ ``in order from $\partial_-Y_k'$ to $\partial_+Y_k'$".  That is, we label the piece of $\cR_{0,1,r}$ in $Y_k'$ containing $\partial_- Y_k'$ with $1$, and recursively label the piece of $\cR_{0,1,r}$ in $Y_k'$ by $i$ if it nontrivially intersects the piece labeled $i-1$. See \Cref{fig:piecelabeling} for an example of this labeling in the neighborhood of an end.

\begin{figure}[ht!]
	    \centering
	    \def\svgwidth{\textwidth}
		    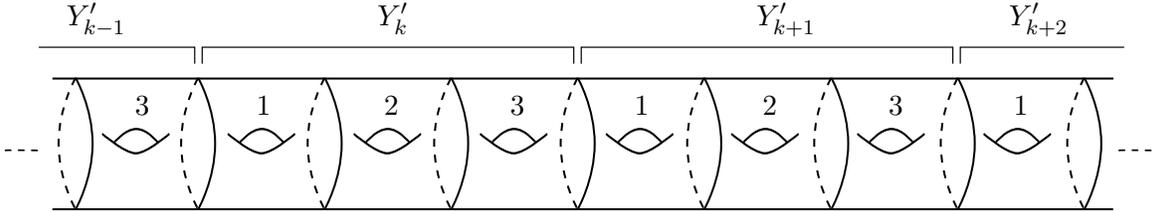
		    \caption{An example of the labeling in \Cref{lem:engulf} for $\cR_{0,1,r}$ engulfing $\cR_{g,3,r}$. Note that the shift map to the right will induce the cyclic permutation $1 \mapsto 2 \mapsto 3 \mapsto 1$.}
	    \label{fig:piecelabeling}
\end{figure}

The sequence of consecutive pieces of $\cR_{0,1,r}$ that exit any end have labels $1,\ldots,h$, which repeat periodically. In some neighborhood of each end any element of $PB(0,1,r)$ acts as a shift and therefore cyclically permutes the labels modulo $h$ in this neighborhood.  This defines an action of $PB(0,1,r)$ on the disjoint union of $r$ sets of $h$ elements. This action fixes each set of $h$ elements setwise and the kernel of this action is precisely $PB(g,h,r)$, which is thus a finite index subgroup of $PB(0,1,r)$, as required.
\end{proof}

\begin{lemma} \label{lem:engulfing shift}
    For any integers $g,n,m \geq 0$, $h \geq 1$, and $r \geq 2$, there are rigid structures $\cR_{g,h,r}$, $\cR_{g+n,h,r}$, and $\cR_{g+nr+mh,h,r}$ all engulfed by $\cR_{0,1,r}$ so that $PB(g,h,r) = PB(g+n,h,r)$ and $B(g,h,r) = B(g+nr+mh,h,r)$.
\end{lemma}
\begin{proof}
    First construct any engulfing of $\cR_{g,h,r}$ as described above, so that $B(g,h,r) < B(0,1,r)$.  We construct the core of $\cR_{g+n,h,r}$ from that of $\cR_{g,h,r}$ by adding any $n$ pieces of $\cR_{0,1,r}$ to obtain a connected surface, and then the rest of the rigid structure is determined.  Note that the labeling of pieces of $\cR_{0,1,r}$ by $1,\ldots,h$ in each piece of $\cR_{g+n,h,r}$, as in the proof of \Cref{lem:engulf}, differs from the labeling in $\cR_{g,h,r}$ by a shift, modulo $h$, in each end.  In particular, the kernel of the action of $PB(0,1,r)$ on $\{1,\ldots,h\}$ for each end is the same for each of the two labelings.  Since the kernel of this action is $PB(g,h,r)$ and $PB(g+n,h,r)$ for the original and new labeling, respectively, we see that $PB(g,h,r) = PB(g+n,h,r)$.

    For the full group $B(g,h,r)$, we note that the labeling of the pieces of $\cR_{0,1,r}$ is invariant on any piece where an element of $B(g,h,r)$ is rigid.  We may therefore construct $\cR_{g+nr,h,r}$ as in the previous paragraph, except now taking care to add the same number of pieces of $\cR_{0,1,r}$, modulo $h$, towards each of the $r$ ends.  More precisely, we may add a single piece of $\cR_{0,1,r}$ to each boundary component of the core of $\cR_{g,h,r}$ which shifts the labels by exactly $1$, modulo $h$, and produces a $(g+r,h)$--rigid structure. Repeating this $n$ times results in $\cR_{g+nr,h,r}$.

    Next we can also add $h$ pieces to the core of $\cR_{g+nr,h,r}$ by adding all of them to a boundary component of the core. In this way the labels on pieces towards each end will not change at all. Repeating this $m$ times results in the required $\cR_{g+nr+mh,h,r}$. Note that, as in the pure case, $B(0,1,r)$ acts on each of these labels, defining two actions on the disjoint union of $r$ sets of $h$ elements. Unlike for the action $PB(0,1,r)$, this action may permute the disjoint sets of labels. However, we again have that the kernel of the two actions is the same, and these groups are precisely $B(g,h,r)$ and $B(g+nr+mh,h,r)$ for the old and new labeling, respectively, and so we see that $B(g,h,r) = B(g+nr+mh,h,r)$. 
\end{proof}

\section{Proof of Theorem~\ref{thm:main}}

In this section, we prove \Cref{thm:main}.  The proof is divided into subsections, one for each part of the theorem.

\subsection{Commensurability}

We start with the following.

\begin{proof}[Proof of \Cref{thm:main}(1).]
By \Cref{lem:engulf}, it suffices to prove that $B(0,1,r)$ is commensurable to $B(0,1,r')$ if and only if $r=r'$.  If $r=r'$, then these groups are equal.  We therefore assume $r \neq r'$, and prove that $B(0,1,r)$ and $B(0,1,r')$ are not commensurable. 

According to \Cref{prop:compactly supported} any finite quotient of $B(0,1,r)$ factors through $\ZZ^{r-1} \rtimes \Sym(r)$ and any finite quotient of $B(0,1,r')$ factors through $\ZZ^{r'-1} \rtimes \Sym(r')$. In particular, the suprema of ranks for abelianizations of finite index subgroups of $B(0,1,r)$ and $B(0,1,r')$ are $r-1$ and $r'-1$, respectively. Since $r \neq r'$,  the groups $B(0,1,r)$ and $B(0,1,r')$ are not commensurable. 
\end{proof}


\subsection{Pure subgroups}

We first check that any isomorphism between surface Houghton groups comes from conjugation.

\begin{proposition}\label{prop:commen}
    If $B(g,h,r) \cong B(g',h',r')$, then $r=r'$ and they are conjugate inside $\Map(\Sigma_r)$. The same also holds if $PB(g,h,r) \cong PB(g',h',r')$. 
\end{proposition} 

This will follow from a minor modification to the proof of the analogous statement for full mapping class groups of infinite-type surfaces (and their pure subgroups) by Bavard-Dowdall-Rafi \cite{BDR2020}[Theorem 1.1].  Their theorem states, more precisely, that an isomorphism of finite index subgroups of mapping class groups (or pure mapping class groups) of infinite type surfaces is given by conjugation by a homeomorphism of the underlying surfaces. 



\begin{proof}[Proof of \Cref{prop:commen}]
    First we apply \Cref{thm:main}(1) to see that $r = r'$ so that both $B(g,h,r)$ and $B(g',h',r)$ can be realized as subgroups of $\Map(\Sigma_{r})$.

    The proof is exactly the same as that of \cite{BDR2020}[Theorem 1.1], with the following modifications.  First, a key ingredient is \cite{BDR2020}[Proposition~4.2], which gives an algebraic characterization of the elements of finite support (equivalently, compact support for $\Sigma_r$). In our case, this algebraic characterization is supplied by \Cref{prop:compactly supported}, which  implies that the intersection of all finite index subgroups is precisely the set of compactly supported elements.  From here, we recall that \cite{BDR2020}[Definition~4.7] lists some algebraic conditions that by \cite{BDR2020}[Corollary~4.8] characterize so called ``generating twists" (we note that, when the finite index subgroup is the entire mapping class group or pure subgroup, these are simply the positive and negative Dehn twists about simple closed curves).  We replace their condition (1) in \cite{BDR2020}[Definition~4.7] characterizing $f$ of finite support with the condition that $f$ is in the intersection of all finite index subgroups of $B(g,h,r)$.  Condition (2) and condition (4) of their definition remain the same, while we add to their condition (3) the requirement that we only consider elements of the centralizer that have compact support (determined algebraically by our modified condition (1)). With these changes, Bavard-Dowdall-Rafi's proof of \cite{BDR2020}[Theorem 1.1] now goes through without any changes to prove that an isomorphism $B(g,h,r) \cong B(g',h',r')$ is induced by conjugation in $\Map(\Sigma_r)$, where $r=r'$, and similarly if $PB(g,h,r) \cong PB(g',h',r')$.
\end{proof}


\begin{proof}[Proof of \Cref{thm:main}(2).]
    Here we consider $PB(g,h,r)$ and $PB(g',h',r')$, and observe that by \Cref{thm:main}(1), we may assume $r=r'$. Thus we will check that $PB(g,h,r)$ and $PB(g',h',r)$ are isomorphic if and only if $h=h'$. 

    We first suppose that $h \neq h'$. Let $\Phi:\PMap(\Sigma_{r}) \rightarrow \ZZ^{r-1}$ be the surjective homomorphism defined in \cite{APV2020}, which ``counts" the shifting of genus. Therefore, we see that $\Phi(PB(g,h,r)) = (h\ZZ)^{r-1}$ and $\Phi(PB(g',h',r)) = (h'\ZZ)^{r-1}$. Thus, since $h \neq h'$ we see that each of these groups has a different image under $\Phi$. In particular, they cannot be conjugate in $\Map(\Sigma_{r})$ and therefore cannot be isomorphic, by \Cref{prop:commen}. 

    Now suppose that $h = h'$. By \Cref{lem:engulfing shift}, we can choose rigid structures engulfed by a fixed structure $\cR_{0,1,r}$ so that $PB(g,h,r) = PB(g',h,r)$ inside the group $PB(0,1,r)$. 
\end{proof}

\subsection{Isomorphisms}

In this section, we prove the last part of \Cref{thm:main}.  The proof uses the structure of certain torsion elements, which requires an auxiliary construction we now describe.

Let $k>0$ be an even integer, $r \geq 2$ any integer, and consider any $g \geq 0$ of the form
    \[ g = 1 +\tfrac{k}2(r-1) + nr.\]
for some integer $n \geq -1$.
We construct an order $r$ homeomorphism of the compact surface $S_{g,r}$ of genus $g$ with $r$ boundary components having one orbit of boundary components and exactly $k$ fixed points.
For the construction, first consider the orbifold that is a disk with $k$ orbifold points of order $r$.  We denote this by $\mathcal O(k,r)$, and note its orbifold Euler characteristic is given by $\chi(\mathcal O(k,r)) = 1 - k(1-\frac{1}r)=1 - k + \tfrac{k}r$.
There is a homomorphism of the orbifold fundamental group
\[ \pi_1^{orb}(\mathcal O(k,r)) \cong \underbrace{\mathbb Z/r \mathbb Z \ast \cdots \ast \mathbb Z/r\mathbb Z}_k  \to \mathbb Z/r \mathbb Z\] sending loops around half the orbifold points to $1$ and the other half to $-1$.  The degree $r$, regular orbifold-cover corresponding to the kernel is a surface $S_{g_0,r}$ with Euler characteristic $\chi(S_{g_0,r}) = r+k-rk$.  This surface has $r$ boundary components, since the loop around the boundary of the disk is in the kernel of the homomorphism, and thus the genus $g_0$ of $S_{g_0,r}$ is $g_0 = 1+\tfrac{k}2(r-1)-r$.

The generator, $\rho_0 \colon S_{g_0,r} \to S_{g_0,r}$, of the covering group is a homeomorphism with the required properties for the case $n=-1$.  For all $n \geq 0$, we can glue on a genus $n+1$ surface with two boundary components to each of the $r$ boundary components to obtain the surface $S_{g,r}$, where $g = 1+\frac{k}2(r-1)+nk$, and extend $\rho_0$ to the required homeomorphism $\rho$ of $S_{g,r}$ in the obvious way.  This proves half of the following:

\begin{lemma} \label{lem:finite order}
Let $k > 0$ be even, $r \geq 2$, and $g \geq 0$ be integers.  Then there exists an order $r$, orientation preserving homeomorphism of $S_{g,r}$ with one orbit of boundary components, exactly $k$ fixed points, and no other periodic points (of period less than $r$) if and only if $g = 1 +\tfrac{k}2(r-1) + nr$ for some integer $n \geq -1$.
\end{lemma}
\begin{proof}  We have already constructed the required homeomorphism under the given conditions on $g$. 

To prove the converse, suppose $\rho \colon S_{g,r} \to S_{g,r}$ is a self-homeomorphism with the specified properties, then the quotient $\mathcal O = S_{g,r}/\langle \rho \rangle$ is an orbifold with orbifold Euler characteristic $\chi(\mathcal O)= \tfrac1r (2-2g-r)$.  This orbifold has $1$ boundary component and $k$ orbifold points of order $r$. 
 Denoting its genus by $g_0$ we then have
\[\tfrac1r(2-2g-r) = 1-2g_0 - k(1-\tfrac1r),\]
and thus
\[ g = 1+\tfrac{k}2(1-r)+r(g_0-1).\]
Since $g_0 \geq 0$, this completes the proof.
\end{proof}

\begin{proof}[Proof of \Cref{thm:main}(3).]
First, suppose that $(g,h,r)$ and $(g',h,r)$ satisfy condition (3) from the statement of the theorem, and let $n,m$ be integers so that $g' = g + nh + mr$.  It follows from \Cref{lem:engulfing shift} that we may choose rigid structures engulfed by a fixed structure $\cR_{0,1,r}$ so that $B(g',h,r) = B(g,h,r)$ inside $B(0,1,r)$, proving one implication of part (3).

If $B(g,h,r) \cong B(g',h',r')$, then in particular they are commensurable so we have $r = r'$.  By \Cref{lem:engulfing shift}, we may view them both as subgroups of $B(0,1,r)$ by choosing rigid structures engulfed by some $\cR_{0,1,r}$.  By \Cref{prop:compactly supported}, it follows that $PB(g,h,r) \cong PB(g',h,r)$, since the pure subgroup of each is the unique normal subgroup with quotient ${\rm Sym}(r)$.  Therefore, $h=h'$ by part (2).  Furthermore, by \Cref{prop:commen}, we can assume that the isomorphism $B(g,h,r) \cong B(g',h,r)$ is induced by conjugation inside $\Map(\Sigma_r)$. 

Now let $W_0 \subset \Sigma_r$ be a suited subsurface for $\cR_{g,h,r}$ of genus $g_0$ with
\[ g_0 = 1+ \tfrac{k}2(r-1) + n r\]
for some even integer $k \geq 0$ and integer $n \geq -1$.  To do this, first note that the genera of suited subsurfaces are precisely the integers $g+mh$, as $m$ ranges over all positive integers.  Next, choose even $k \geq 0$ with $g \equiv 1 - \tfrac{k}2 \mbox{ mod } r$.
Thus there is an integer $n_0$ so that
\[ g = 1-\tfrac{k}2 + (n_0+\tfrac{k}2)r = 1 + \tfrac{k}2(r-1) + n_0r.\]
Then choose a large enough positive integer $m_0$ so that $n = n_0+m_0h \geq -1$, and thus
\[ g_0 = g + m_0hr = 1+\tfrac{k}2(r-1) + (n_0+m_0h)r\]
is the genus of a suited subsurface.  Let $m = m_0h$ here.

Now consider the element $\rho$ of order $r$ from \Cref{lem:finite order} applied to $W_0$.  Since $\rho$ is simply a cyclic permutation on the boundary components, we can extend $\rho$ to a finite order element $\hat \rho$ on all of $\Sigma_r$ with $\hat \rho \in B(g,h,r)$.  This homeomorphism has precisely the same fixed/periodic point data as $\rho$.  The image of $\hat \rho$ in $B(g',h,r)$ by the conjugation isomorphism is an element $\hat \rho'$ of order $r$ with the same fixed point/period data as well.  Since $\hat \rho'$ is rigid for $\cR_{g',h,r}$, there is a $\cR_{g',h,r}$--suited subsurface $W_1\subset \Sigma_r$ so that $\hat \rho'$ is rigid outside $W_1$.  Adding pieces to $W_1$ if necessary, we can assume $W_1$ is invariant by $\hat \rho$.  This suited subsurface has genus $g_1 = g' + m'h$ for some integer $m' \geq 0$ (since it is the union of the core and some number of pieces of $\cR_{g',h,r}$).  All fixed points of $\hat \rho$ are in $W_1$, and so it has exactly $k$ fixed points and no periodic points of order less than $r$. 
 Thus by \Cref{lem:finite order}, we have
 \[ g_1 = 1+ \tfrac{k}2(r-1) + n'r\]
for some $n' \geq -1$.  Thus
\[ g' +m'h = 1+\tfrac{k}2(r-1) +n'r = g_0-nr+n'r = g+mh +(n'-n)r. \]
Therefore $g' \equiv g + (m-m')h \mbox{ mod } r$, as required.
\end{proof}

\bibliography{bib}
\bibliographystyle{plain}

\end{document}